\documentclass[12pt]{amsart}
 \usepackage[left=3cm, right=3cm, top=3.5cm, bottom=4cm]{geometry} 
\usepackage[utf8]{inputenc}
\usepackage{amssymb}
\usepackage{amscd}
\usepackage{amsmath}
\usepackage{enumerate}
\usepackage{comment}
\usepackage{dsfont}
\usepackage{amsthm}
\usepackage{url}
\usepackage{ mathrsfs }
\usepackage{textcomp}
\usepackage[all]{xy}
\usepackage{tikz-cd}
\usepackage{xcolor}
\definecolor{darkgrey}{rgb}{0.4,0.4,0.5}
\usepackage{hyperref}
\hypersetup{
    colorlinks=true,
    linkcolor=blue,
    citecolor=darkgrey
}

\usepackage{url}

\usepackage{multirow}

\newtheorem{theorem}{Theorem}[section]
\newtheorem{lemma}[theorem]{Lemma}
\newtheorem{question}[theorem]{Question}

\newtheorem{conjecture}[theorem]{Conjecture}
\newtheorem{corollary}[theorem]{Corollary}
\newtheorem{proposition}[theorem]{Proposition}
\newtheorem{definition}[theorem]{Definition}
\newtheorem{proposition-definition}[theorem]{Proposition-Definition}

\theoremstyle{remark}
\newtheorem{remark}[theorem]{Remark}

\numberwithin{equation}{section}

\newcommand\EatDot[1]{}

\newcommand{\unr}{{\mathrm{unr}}}

\newcommand{\Gal}{{\mathrm{Gal}}}

\newcommand{\Coker}{{\mathrm{Coker}}}

\newcommand{\cris}{{\mathrm{cris}}}

\newcommand{\Proj}{\mathrm{Proj}}
\newcommand{\lalg}{\mathrm{lalg}}
\newcommand{\sm}{\mathrm{sm}}
\newcommand{\an}{\mathrm{an}}
\newcommand{\la}{\mathrm{la}}
\newcommand{\Sym}{\mathrm{Sym}}

\newcommand{\Hom}{\mathrm{Hom}}
\newcommand{\Lie}{\mathrm{Lie}}

\newcommand{\Ind}{\mathrm{Ind}}
\newcommand{\Q}{{\mathbb Q}}
\newcommand{\Z}{{\mathbb Z}}

\newcommand{\Ext}{\mathrm{Ext}}
\newcommand{\Spec}{\mathrm{Spec}}
\newcommand{\irr}{\mathrm{irr}}

\newcommand{\OO}{{\mathcal O}}
\newcommand{\XX}{\mathds{X}}
\newcommand{\GG}{\mathds{G}}

\DeclareMathOperator{\val}{val}
\newcommand{\charvar}{\widehat{\mathscr{T}}}
\newcommand{\robba}{\mathscr{R}}
\newcommand{\scriptS}{\mathscr{S}}
\newcommand{\scriptL}{\mathscr{L}}
\newcommand{\frob}{\varphi}
\newcommand{\phigamma}{(\varphi, \Gamma)}
\newcommand{\rig}{\mathrm{rig}}
\newcommand{\Ddaggerrig}{D_{\rig}^{\dagger}}
\newcommand{\scriptC}{\mathcal{C}}
\newcommand{\scriptD}{\mathcal{D}}
\newcommand{\CA}{\mathcal{C}^{\an}}
\newcommand{\CAXU}{\CA(\XX,U)}
\newcommand{\CAGU}{\CA(\GG,U)}
\newcommand{\UGan}{U_{\GG-\an}}
\newcommand{\Gbf}{\mathbf{G}}
\newcommand{\Gnnot}{\GG(n)^{\circ}}
\newcommand{\DanGnnotL}{\mathcal{D}^{\an}(\Gnnot,L)}
\newcommand{\UU}{\mathds{U}}
\newcommand{\frakRKKnot}{\mathfrak{R}_{K/K_0}}
\newcommand{\resRKKnot}{\mathfrak{R}_{K/K_0}}

\newcommand*\closure[1]{\overline{#1}} % (or \bar{#1})
\newcommand{\absolute}[1]{\left\lvert#1\right\rvert}

\begin{document}
\title[Rigid analytic geometry and $p$-adic Langlands]{Rigid analytic vectors of crystalline representations arising in $p$-adic Langlands}

\author{Jishnu Ray}
\address{Department of Mathematics, The University of British Columbia}
\curraddr{Room 121, 1984 Mathematics Road\\
	Vancouver, BC\\
	Canada V6T 1Z2}
\email{jishnuray1992@gmail.com}

\thanks{This work is supported by PIMS-CNRS postdoctoral research grant from the University of British Columbia.
}

\begin{abstract}
Let $\mathbf{B}(V)$ be the admissible unitary $GL_2(\Q_p)$-representation associated to two dimensional crystalline Galois representation $V$ by $p$-adic Langlands constructed by Breuil. Berger and Breuil conjectured an explicit description of the locally analytic vectors $\mathbf{B}(V)_{\la}$ of  $\mathbf{B}(V)$ which is now proved by Liu. Emerton recently studied $p$-adic representations from the viewpoint of rigid analytic geometry.  In this article, we consider certain rigid analytic subgroups of $GL(2)$ and give an explicit description of the rigid analytic vectors in $\mathbf{B}(V)_{\la}$. In particular, we show the existence of rigid analytic vectors inside $\mathbf{B}(V)_{\la}$ and prove that its non-null. This gives us a rigid analytic representation (in the sense of Emerton) lying inside the locally analytic representation $\mathbf{B}(V)_{\la}$.  
\end{abstract}

\maketitle

%\tableofcontents

\section{Introduction}
Let $p>2$ be a prime and $L$ be a finite extension over $\Q_p$. Breuil initiated the construction of $p$-adic local Langlands for $GL_2(\Q_p)$. To any two dimensional, crystalline Galois $L$-representation  $V$ he attached a $p$-adic admissible unitary representation of  $GL_2(\Q_p)$ which we denote by $\mathbf{B}(V)$. Further, the main work of Colmez was to extend this correspondence to any two dimensional $L$-representation of the absolute Galois group of $\Q_p$. When $V$ is crystalline, Berger and Breuil formulated a conjecture on the locally analytic vectors $\mathbf{B}(V)_{\la}$ of the $L$-Banach representation $\mathbf{B}(V)$. The representation $\mathbf{B}(V)_{\la}$ is an admissible locally analytic representation in the sense of Schneider and Teitelbaum \cite{Algebras}. This conjecture was proved by Liu in \cite[Thm. 4.1]{Liu}. Colmez and Dospinescu generalized this to all $GL_2(\Q_p)$-representations appearing in the $p$-adic Langlands \cite{ColDos}. In \cite[Sec. 3.3]{Emertonbook}, Emerton studies rigid analytic vectors of $p$-adic representations from the viewpoint of rigid analytic geometry \cite{Bosch1}. Then, he builds up a notion of an admissible rigid analytic representation which is completely analogous to the original construction of Schneider and Teitelbaum for an admissible locally analytic representation.

Let $I(1)$ be the pro-$p$ Iwahori subgroup of $GL_2(\Z_p)$, that is matrices of the form 
$\left( \begin{matrix} 1+\Z_p & p\Z_p \\  \Z_p & 1+\Z_p \end{matrix} \right).$  Let $Q$ be the intersection of the standard upper triangular Borel subgroup of $GL_2(\Z_p)$ with the pro-$p$ Iwahori.  The pro-$p$ Iwahori $I(1)$ can be considered as $\Z_p$-points of a rigid analytic group. 
 In \cite[Sec. 3.3]{Clozel_global_II_published}, Laurent Clozel showed the existence of $I(1)$-rigid analytic vectors inside the principal series representation, which is the representation induced by a rigid analytic character $\chi$ of the Borel subgroup $Q$ to $I(1)$. This gave him an admissible rigid analytic representation, in the sense of Emerton, of the rigid analytic pro-$p$ Iwahori subgroup. Furthermore, he found conditions on the character $\chi$ such that his rigid analytic principal series is irreducible \cite[Thm. 3.7]{Clozel_global_II_published}. This result of Clozel is later generalized by the author of this article for $GL_n(\Z_p)$ \cite[Thms. 3.8, 3.9]{Rayglobal}. Clozel once asked the author the following question.
 \begin{question}[by Clozel]
 	Start with a two dimensional, Galois $L$-representation  $V$. By work of Fontaine, we can attach to it a rank two  étale $(\varphi, \Gamma)$-module $D$ over the Robba ring $\robba_L$ (cf. Section \ref{sec:trianguline}). Then Colmez and others have constructed the corresponding $GL_2(\Q_p)$-representation $\Pi(D)$ which is of type $\mathbf{\Pi}(D) \cong ( D \boxtimes \mathbf{P}^1)/(D^\natural \boxtimes \mathbf{P}^1) $
 	(cf. Section \ref{sec:padicLang}) and set the $p$-adic Langlands as $V \rightarrow \Pi(D)$. 
 	
 	Does there exist rigid analytic vectors inside $\Pi(D)?$ If so, how to compute them and construct a rigid analytic representation, in the sense of Emerton?
 \end{question}
This question is answered positively in this article for crystalline representations $V$.  When $V$ is not crystalline, this question is still open. The rigid analytic groups we will consider are the so-called `good open congruence subgroups'   which have  the property that $\Gnnot (\mathbb{C}_p)= 1+\varpi^n(\mathfrak{m}_{\mathbb{C}_p})$ where $\varpi^n \in \mathcal{O}_{\mathbb{C}_p}$ is a uniformizer  and $\mathfrak{m}_{\mathbb{C}_p}$ is the maximal ideal. The group $\Gnnot$ is a rigid analytic subgroup of $GL_2^{\rig}$ which is a rigid analytic group associated to the algebraic group $GL_2$.  For $m$ a positive integer, let $\GG(m)$ be the rigid analytic group whose $\Z_p$-points is the $m$-th principal congruence subgroup of $GL_2(\Z_p)$. Then $\Gnnot$ is the  rigid analytic group which is the union of the rigid analytic groups $\GG(m)$ for all $m>n$.

When $V$ is crystalline, the theme of the paper is to  compute  the subspace $\big(\mathbf{B}(V)_{\la}\big)_{\Gnnot - \an}$ of $\Gnnot$-analytic vectors of the locally analytic $GL_2(\Q_p)$-representation $ \mathbf{B}(V)_{\la}$ associated to $V$ by $p$-adic Langlands of Berger and Breuil. We give an explicit presentation of $\big(\mathbf{B}(V)_{\la}\big)_{\Gnnot - \an}$ and show that its non-null. 

Our main Theorem is the following.
\begin{theorem}[see Thm. \ref{thm:mainAnalytic}]
	Suppose $V=V(\alpha,\beta)$ is crystalline, irreducible, Frobenius semisimple Galois representation with distinct Hodge-Tate weights. Let $\mathbf{B}(V)_{\la}$ be the locally analytic vectors of the $GL_2(\Q_p)$-representation  $\mathbf{B}(V)$ associated to $V$ via $p$-adic local Langlands constructed by Berger and Breuil. Then,
	$$\big(\mathbf{B}(V)_{\la}\big)_{\Gnnot-\an} \cong \Coker\Big( \pi(\beta)_{\Gnnot-\an} \hookrightarrow LA(\alpha)_{\Gnnot-\an} \oplus LA(\beta)_{\Gnnot-\an}  \Big)  $$
	is a rigid-analytic $\Gnnot(\Z_p)$-representation. Furthermore,
	\begin{enumerate}[(i)]
		\item $\pi(\beta)_{\Gnnot-\an} \cong \Sym^{k-2}L^2 \otimes \varprojlim_{m>n}C(\beta)_{m}$,
		\item   $ LA(\alpha)_{\Gnnot-\an} \cong GA(\alpha)$  and $ LA(\beta)_{\Gnnot-\an} \cong GA(\beta)$.
	\end{enumerate}
	In particular, it follows that $\big(\mathbf{B}(V)_{\la}\big)_{\Gnnot-\an} \neq 0$.
\end{theorem}
Here $LA(\alpha)$  or $LA(\beta)$ are locally analytic principal series (cf. \eqref{eq:LAalpha}) and $\pi(\beta)$ is a closed $GL_2(\Q_p)$-equivariant subspace of locally algebraic functions inside $LA(\beta)$ (cf. \eqref{embeddingsClosed} and \eqref{pibeta}). Here, $C(\beta)_m$, $GA(\alpha)$ and $GA(\beta)$ are rigid analytic functions whose descriptions are explicitly given in Theorem \ref{thm: Amalytic LA alpha beta} and Corollary \ref{cor:pibetaAnalytic}. Our main Theorem tells us that the representation of $\Gnnot$-analytic vectors of  $\mathbf{B}(V)_{\la}$ is a non-trivial quotient of the direct sum of two principal series representations.

There are special advantages (see Section \ref{sec: whySigmaAffinoid}) of considering this group $\Gnnot$, which is a $\sigma$-affinoid space (cf. Section \ref{sec:rigidbasic}) rather than the pro-$p$ Iwahori considered originally by Clozel in \cite{Clozel_global_II_published}. One of them is the fact that  its distribution algebra $\DanGnnotL$ is a coherent ring (cf. Section \ref{sec:wideopen}). The other is the fact that  the (strong) dual of $\big(\mathbf{B}(V)_{\la}\big)_{\Gnnot-\an} $ is a \textit{finitely presented} module over the distribution algebra $\DanGnnotL$ (cf. Proposition \ref{prop:finitepresented}). On the contrary, nothing is know about the ring theoretic structure of the distribution algebra of the affinoid pro-$p$ Iwahori and so we don't know if there is a theory of coadmissible modules over them, analogous to the case of Schneider and Teitelbaum for admissible locally analytic representations \cite{Algebras}. (Coadmissible modules arise by taking the strong dual of admissible representations). We do have a theory of coadmissible modules over $\DanGnnotL$, thanks to the work of Patel, Schmidt and Strauch \cite[Section 5]{PSS2}. Furthermore, there are several pathologies concerning the distribution algebra of the pro-$p$ Iwahori that Clozel explains in the appendix of his paper \cite{Clozel_global_II_published}. 

The organization of the article is as follows. In Section \ref{sec:basicSetup}, we introduce the basic notations and recall facts about trianguline representation and Colmez's construction of $p$-adic Langlands. Then, in Section \ref{sec:locallyanalytic basic}, we give the description of $\mathbf{B}(V)_{\la}$ after recalling some basics of locally analytic representation theory. In Sections \ref{sec:rigidbasic} and \ref{sec:wideopen}, following Emerton, we define rigid analytic representation and show its connection with locally analytic representation of Schneider and Teitelbaum. Further, in these Sections, we also collect crucial results on rigid analytic representations from the literature which are useful for the proof of our main Theorem. The notion of rigid analytic representations, due to Emerton, is relatively new and therefore we find it valuable to include a short exposition on it which explains our main theorem. A reader familiar with $p$-adic Langlands and analytic representations can completely ignore reading these Sections and move on to Section \ref{sec:mainheart}. However, we chose to include them so that our article is self contained. Our main Theorem is proved in Section \ref{sec:mainheart}. Moreover, we also include some open questions on Langlands base change of rigid analytic vectors in Section \ref{sec:basechange}.

\section{Basic setup}\label{sec:basicSetup}
\subsection{Notations}
 Let $G_{\Q_p}$ be the absolute Galois group of $\Q_p$ and $W_{\Q_p}$ be the Weil group which is dense in $G_{\Q_p}$. Let $\chi: G_{\Q_p} \rightarrow \Z_p^{\times}$ be the cyclotomic character. Let $F_\infty=\Q_p(\mu_{p^{\infty}})$ and $H_{\Q_p}= \ker \chi = \Gal(\closure{\Q}_p/F_\infty).$ We can view $\chi$ as an isomorphism from $\Gamma=G_{\Q_p}/H_{\Q_p}=\Gal(F_\infty/\Q_p)$ to $\Z_p^{\times}$. Let $\val_p(-)$ denote the $p$-adic valuation on $\closure{\Q}_p$ normalized by $\val_p(p)=1$.  Let $ \charvar(L)$  be the set of continuous characters $\delta: \Q_p^{\times} \rightarrow L^\times$. Following Colmez \cite{Colmeztri}, we remark that this notation $\charvar(L)$ is justified by the fact that it is the $L$-rational point of a character variety $\charvar$. We define the weight $w(\delta)$ of $\delta$ by the formula $w(\delta)=\frac{\log \delta(u)}{\log u}$ where $u \in \Z_p^\times$ is not a root of unity. By class field theory, we can view an element of $\charvar(L)$ as a continuous character of the Weil group $W_{\Q_p}$. Furthermore, if $\delta$ is a unitary character (having values in $\mathcal{O}_L^{\times}$), then one can uniquely extend $\delta$ to a continuous character of $W_{\Q_p}$. Then, one can reinterpret $w(\delta)$ as generalized Hodge-Tate weight of $\delta$. We denote $x \in \charvar(L)$ to be the character induced by the inclusion of $\Q_p$ in $L$, and $\absolute{x}$ the character sending $x \in \Q_p^\times$ to $p^{-\val_p(x)}$. Then $x\absolute{x}$ is unitary corresponding to the cyclotomic character $\chi$. 
\subsection{Trianguline representations}\label{sec:trianguline}
Let $\robba_L$ denote the Robba ring over $L$ with the Frobenius $\varphi$-action and $\Gamma$-action given by $\frob(T)=(1+T)^p-1$ and $\gamma(T)=(1+T)^{\chi(\gamma)}-1$ for any $\gamma \in \Gamma$. For $\delta \in \charvar(L)$, let $\robba_L(\delta)$ be the Robba twisted by the character $\delta$. It has $\robba_L$-basis $e$ such that $\frob, \Gamma$-actions are given by $\frob(xe)=\delta(p)\frob(x)e$ and $\gamma(xe) = \delta(\chi(\gamma)) \gamma(x) e$ for any $x \in \robba_L$. Colmez shows that if $M$ is a rank one $\phigamma$-module over $\robba_L$, then $M \cong \robba_L(\delta)$ for some $\delta$ \cite[Prop. 3.1]{Colmeztri}. Recall that a $\phigamma$-module over $\robba_L$ is called triangulable if it can be expressed as successive extensions of rank one $\phigamma$-modules over $\robba_L$. An $L$-representation $V$ of $G_{\Q_p}$ is called trianguline if its associated $\phigamma$-module, denoted by $\Ddaggerrig(V)$, over $\robba_L$ is triangulable, i.e. it fits into a short exact sequence $$ 0 \rightarrow \robba_L(\delta_1) \rightarrow \Ddaggerrig(V) \rightarrow \robba_L(\delta_2) \rightarrow 0.$$ The representation $V$ is then determined by a triple $(\delta_1,\delta_2, h)$ where $h \in \Ext^1(\robba_L(\delta_2), \robba_L(\delta_1)),$ the extension corresponding to $\Ddaggerrig(V)$. The generalized Hodge-Tate weights of $V$ will then be $w(\delta_1)$ and $w(\delta_2)$ and it follows that $\val_p(\delta_1(p)) + \val_p(\delta_2(p)) =0$. Colmez studied $ \Ext^1(\robba_L(\delta_2), \robba_L(\delta_1))$ and showed that for any two characters $\delta_1,\delta_2 \in \charvar(L)$,  $ \Ext^1(\robba_L(\delta_2), \robba_L(\delta_1))$ is an $L$-vector space of dimension $1$ unless $\delta_1\delta_2^{-1}$ is of the form $x^{-i} (i \geq 0)$, or of the form $\absolute{x}x^i (i \geq 1)$; in the last two cases, $ \Ext^1(\robba_L(\delta_2), \robba_L(\delta_1))$ is of dimension $2$ and the associated projective space is naturally isomorphic to $\mathbf{P}^1(L)$ \cite[Prop. 8.2]{ColmezIII}. Let $\scriptS$ be the set of $(\delta_1, \delta_2, \scriptL)$, where $(\delta_1, \delta_2)$ is an element of $\charvar(L) \times \charvar(L)$ and $\scriptL \in \Proj( \Ext^1(\robba_L(\delta_2), \robba_L(\delta_1))) $, the later space being identified with $\mathbf{P}^{0}(L)=\{\infty\} $ (resp. $\mathbf{P}^1(L)$)  if $ \Ext^1(\robba_L(\delta_2), \robba_L(\delta_1))$ is of dimension $1$ (resp. $2$).  Take $s=(\delta_1,\delta_2, \scriptL)$ which determines an extension $D(s)$ of twisted Robba rings upto isomorphism by $\scriptL$. We have a short exact sequence $$0 \rightarrow \robba_L(\delta_1) \rightarrow D(s) \rightarrow \robba_L(\delta_2) \rightarrow 0.$$ We let $\scriptS_{\ast}$ to be the subset of triples  $\scriptS$ consisting of $s$ such that $\val_p(\delta_1(p)) + \val_p(\delta_2(p)) =0$ and $\val_p(\delta_1(p)) >0$ and we define $u(s)= \val_p(\delta_1(p))=-\val_p(\delta_2(p))$ and $w(s)= w(\delta_1) -w(\delta_2)$. In the case when $D(s)$ is étale, we denote by $V(s)$, the $L$-linear representation of $G_{\Q_p}$ such that $\Ddaggerrig(V(s)) =D(s)$. In \cite{Colmeztri}, Colmez partitions $\scriptS_\ast$ into $$\scriptS_\ast = \scriptS_\ast^{\mathrm{ng}} \sqcup \scriptS_\ast^{\mathrm{cris}} \sqcup \scriptS_\ast^{\mathrm{st}} \sqcup \scriptS_\ast^{\mathrm{ord}} \sqcup \scriptS_\ast^{\mathrm{ncl}}. $$
Here the exponents `ng', `cris', `st', `ord', `ncl' refer to `non geometric', `crystalline', `semistable', `ordinary', `non classical' respectively. From these Colmez classified the irreducible ones $$\scriptS_\irr  = \scriptS_\ast^{\mathrm{ng}} \sqcup \scriptS_\ast^{\mathrm{cris}} \sqcup \scriptS_\ast^{\mathrm{st}}, $$
and showed that if $s \in \scriptS_\irr$, then $D(s)$ is étale and $V(s)$ is irreducible. Conversely, if $V$ is irreducible and trianguline, then $V\cong V(s)$ for some $s \in \scriptS_\irr$ \cite[Thm. 0.5(i),(ii)]{Colmeztri}. Throughout this paper we will be interested with the crystalline ones $\scriptS_\ast^{\mathrm{cris}} $. \\

$\scriptS_\ast^{\mathrm{cris}} $ is the set of all $s \in \scriptS_\irr$ such that $w(s)$ is an integer greater or equal to $1$, $u(s) <w(s)$ and $\scriptL = \infty$. \\

We have that $V \in \scriptS_\ast^{\mathrm{cris}} $ if and only if $V$ is twist of an irreducible and crystabelian (becomes crystalline after going to a finite abelian extension of $\Q_p$)  representation. 

\subsection{$p$-adic Langlands for crystalline representations}\label{sec:padicLang}
Let us first briefly recall Colmez's construction of $p$-adic local Langlands for any two-dimensional irreducible $L$-linear representation of $G_{\Q_p}$.  Then, we will provide an explicit description of the $GL_2(\Q_p)$-representation for $V$ crystalline which is due to Berger and Breuil \cite{BerBre}. 
 Let $D=D(s)$ be a rank two, irreducible and étale $\phigamma$-module over $\robba_L$ corresponding to the Galois representation $V=V(s)$. Then Colmez constructs a $GL_2(\Q_p)$-equivariant sheaf on $\mathbf{P}^1=\mathbf{P}^1(\Q_p)$ whose global sections $D \boxtimes \mathbf{P}^1$ give 
 a description of the associated $GL_2(\Q_p)$-representation $\mathbf{\Pi}(D)$. More precisely, he shows that we have an exact sequence $$ 0 \rightarrow \mathbf{\Pi}(D)^{\ast} \otimes \delta_D \rightarrow D \boxtimes \mathbf{P}^1 \rightarrow \mathbf{\Pi}(D) \rightarrow 0,$$ where $\delta_D$ is the character $\chi^{-1}\det(V)$  which is also the central character of $\mathbf{\Pi}(D)$. Here $\mathbf{\Pi}(D)^\ast$ is the topological dual of $\mathbf{\Pi}(D)$ and $\mathbf{\Pi}(D)^{\ast} \otimes \delta_D  \cong D^\natural \boxtimes \mathbf{P}^1$; that is $\mathbf{\Pi}(D) \cong ( D \boxtimes \mathbf{P}^1)/(D^\natural \boxtimes \mathbf{P}^1) $ \cite[Thm.  0.17]{Colmezphi}. Colmez showed that $\mathbf{\Pi}(D)$ is an admissible unitary representation of $GL_2(\Q_p)$  and he sets the $p$-adic local Langlands correspondence as $V \mapsto \mathbf{\Pi}(V):= \mathbf{\Pi}(D(V))$. Let us now suppose that $V$ is crystalline. In the following, we give an explicit construction of the corresponding $GL_2(\Q_p)$-representation due to Berger and Breuil \cite{BerBre}, \cite{BerBreuilNotes}. We can twist $V$ be a suitable power of the cyclotomic character so that its Hodge-Tate weights are $0$ and $k-1$ with $k \geq 1$. Suppose  that $V$ is irreducible with distinct Hodge-Tate weights. By a result of Colmez \cite[Prop. 4.14]{Colmeztri}, we know that $V$ is uniquely determined by a pair of smooth characters of $\Q_p^\times$. Furthermore, suppose that $V$ is Frobenius semisimple, i.e. the Frobenius $\varphi$ on $D_{\cris}(V) = (B_\cris \otimes_{\Q_p} V)^{G_{\Q_p}}$ is semisimple. Here $B_\cris$ is one of Fontaine's period rings and $V$ is considered as a $\Q_p$-linear representation by restriction of scalars. One can check that $D_{\cris}(V)$ is an $L$-vector space and the Frobenius $\varphi: D_\cris(V) \rightarrow D_\cris(V)$ is $L$-linear. Then,  there are two elements (coming from the values at $p$ of the pair of smooth characters associated to $V$), $\alpha, \beta \in \OO_L$ (the ring of integers of $L$) such that $\alpha \neq \beta, 0 < \val_p(\beta) \leq \val_p(\alpha), \val_p(\alpha) + \val_p(\beta) = k-1$ and $D_\cris(V) = D(\alpha,\beta)=Le_\alpha \oplus Le_\beta$ with $\frob(e_\alpha) = \alpha^{-1}e_\alpha$ and $\frob(e_\beta)=\beta^{-1}e_\beta$. The module $D(\alpha,\beta)$ is a filtered $\varphi$-module over $L$ where the filtration is given by 
 \begin{equation*}
 \mathrm{Fil}^i D(\alpha,\beta)=
 \begin{cases}
 D(\alpha, \beta) & \text{if $i \leq -(k-1)$}\\
 L(e_\alpha + e_\beta) & \text{if $-(k-2) \leq i \leq 0$}\\
 0 & \text{if $i >0$}
 \end{cases}       
 \end{equation*}
(cf. Section $7.1$ of \cite{BerBreuilNotes}). Berger and Breuil also showed that the functor $V \mapsto D_{\cris}(V)$ is an equivalence of categories from the category of $L$-linear crystalline representations to the category of (admissible) filtered $\varphi$-modules over $L$ (cf. \cite[Prop. 2.4.5]{BerBre}, see also \cite[Thm. 2.3.1]{BerBreuilNotes}). Let $B(\Q_p)$ be the upper triangular Borel subgroup of $GL_2(\Q_p)$. Following the notations of Berger and Breuil in \cite{BerBreuilNotes}, we define the locally algebraic representations $\pi(\alpha)$ and $\pi(\beta)$ as 
\begin{align}\label{eq:PIAB}
\pi(\alpha)&:=\big(\Ind_{B(\Q_p)}^{GL_2(\Q_p)} \unr(\alpha^{-1}) \otimes x^{k-2}\unr(p\beta^{-1})   \big)_{\lalg} \\
&\cong \Sym^{k-2}L^2 \otimes_L \big(  \Ind_{B(\Q_p)}^{GL_2(\Q_p)} \unr(\alpha^{-1}) \otimes \unr(p\beta^{-1})   \big)_\sm\\
\pi(\beta)&:=\big(\Ind_{B(\Q_p)}^{GL_2(\Q_p)} \unr(\beta^{-1}) \otimes x^{k-2}\unr(p\alpha^{-1})   \big)_{\lalg}  \label{pibeta}\\
&\cong \Sym^{k-2}L^2 \otimes_L \big(  \Ind_{B(\Q_p)}^{GL_2(\Q_p)} \unr(\beta^{-1}) \otimes \unr(p\alpha^{-1})   \big)_\sm,
\end{align}
where $\unr(-)$ is the unramified  character from $\Q_p^\times$  to $L^\times$ given by $\unr(\lambda): y \mapsto \lambda^{\val_p(y)}$ for $\lambda= \alpha^{-1}$ or $p\beta^{-1}$. The subscript $`\sm$' denotes the smooth vectors in the induced principal series. (Note that, to be consistent with notations of Emerton, we have used `sm' as a subscript whereas Berger and Breuil usually writes it as a superscript). We equip $\pi(\alpha)$ (resp. $\pi(\beta)$) with the unique locally convex topology such that the open sets are lattices of $\pi(\alpha)$ (resp. $\pi(\beta)$). (A lattice of an $L$-vector space $U$ is an $\OO_L$-submodule which generates $U$ over $L$). In \cite[Thm. 4.3.1]{BerBre}, Berger and Breuil constructed the universal unitary completion $B(\alpha)/L(\alpha)$ of $\pi(\alpha)$ and a $GL_2(\Q_p)$-equivariant map $\pi(\alpha) \rightarrow B(\alpha)/L(\alpha)$ where 
\begin{equation}
B(\alpha) = \big(\Ind_{B(\Q_p)}^{GL_2(\Q_p)} \unr(\alpha^{-1}) \otimes x^{k-2}\unr(p\beta^{-1})   \big)_{\mathcal{C}^{\val_p(\alpha)}},
\end{equation}
and $L(\alpha)$ is a certain closed subspace of $B(\alpha)$ (see section 7.1 of \cite{BerBreuilNotes}). Similarly interchanging $\alpha$ with $\beta$, we can construct $B(\beta)/L(\beta)$ and realize it as the universal unitary completion of $\pi(\beta)$. Recall that, there exists up to multiplication by a non-zero scalar, a unique non-zero $GL_2(\Q_p)$-intertwining operator  \cite[Sec. 9.1]{BerBreuilNotes}
$$I_{\sm}: \big(  \Ind_{B(\Q_p)}^{GL_2(\Q_p)} \unr(\alpha^{-1}) \otimes \unr(p\beta^{-1})   \big)_\sm \rightarrow \big(  \Ind_{B(\Q_p)}^{GL_2(\Q_p)} \unr(\beta^{-1}) \otimes \unr(p\alpha^{-1})   \big)_\sm.$$
Tensoring with the identity map on $\Sym^{k-2}L^2$, we get a non-zero $GL_2(\Q_p)$-equivariant morphism $I: \pi(\alpha) \rightarrow \pi(\beta)$.  Berger and Breuil showed the following fact \cite[Corollary 7.2.3]{BerBreuilNotes}.
\begin{lemma}\label{lemma:intertwining}
	Assume $0 < \val_p(\alpha) < k-1$, then $B(\alpha)/L(\alpha)$ and $B(\beta)/L(\beta)$ are unitary $GL_2(\Q_p)$-Banach spaces and we have a commutative $GL_2(\Q_p)$-equivariant diagram:

\[	
	\begin{CD}
		B(\alpha)/L(\alpha)     @>{\widehat{I}}>>  B(\beta)/L(\beta)\\
		@AA{~}A       @AA{~}A\\
		\pi(\alpha)    @>I>>  \pi(\beta)
	\end{CD}
\]

where the horizontal maps are isomorphisms and the vertical maps are closed $GL_2(\Q_p)$-equivariant embeddings. Here $\widehat{I}$ is the continuous $GL_2(\Q_p)$-morphism induced from I. 
\end{lemma}
The $p$-adic Langlands of Berger and Breuil associates a non-zero $GL_2(\Q_p)$-Banach representation $\mathbf{B}(V):= B(\alpha)/L(\alpha)$ to $V$. The relation of $\mathbf{B}(V)$ with Colmez's construction is the following. The Banach dual of $\mathbf{B}(V)$ is naturally isomorphic to $D({\check{V}})^\natural \boxtimes \mathbf{P}^1$ as a $L$-Banach space representation of $GL_2(\Q_p)$. Here $\check{V}$ is the contragredient of $V$.

Berger and Breuil conjectured that the locally analytic vectors $\mathbf{B}(V)_{\la}$ of $\mathbf{B}(V)$ have the form  $$\mathbf{B}(V)_{\la} \cong LA(\alpha) \oplus_{\pi(\beta)}LA(\beta),$$
where $LA(\alpha)$ and $LA(\beta)$ are two locally analytic principal series (cf. Section \ref{subsec: locallycrystalline}). This is now proved by Liu \cite{Liu}. 

\begin{comment}
For $m$ a positive integer, let $\GG(m)$ be the rigid analytic group whose $\Z_p$-points is the $m$-th principal congruence subgroup of $GL_2(\Z_p)$. Let $\Gnnot$  be the  rigid analytic group which is the union of the rigid analytic groups $\GG(m)$ for all $m>n$. We discuss these constructions in Section \ref{sec:wideopen} recalling basics  of rigid analytic geometry in Section \ref{sec:rigidbasic}.

Then, in Section \ref{sec:mainheart}, we find an explicit description of the $\Gnnot$-(rigid) analytic vectors of $LA(\alpha)$, $LA(\beta)$ and $\pi(\beta)$.  This is done in Theorem \ref{thm: Amalytic LA alpha beta} and Corollary \ref{cor:pibetaAnalytic}. Finally we use these results to deduce $\Gnnot$-rigid analytic vectors of $\mathbf{B}(V)_{\la}$ in Theorem \ref{thm:mainAnalytic} which is the main result of this article.  Theorem \ref{thm:mainAnalytic} shows the existence  of rigid analytic vectors in $\mathbf{B}(V)_{\la}$, that is, $\big(\mathbf{B}(V)_{\la}\big)_{\Gnnot-\an} \neq 0$.
Finally, in Section \ref{sec:basechange}, we  show how to construct a Langlands base change   of $\big(\mathbf{B}(V)_{\la}\big)_{\Gnnot-\an}$   to a rigid analytic representation of  $\Gnnot(K)$ where $K$ is a finite  unramified extension of $\Q_p$ and discuss few open questions.

\end{comment}

\section{Locally analytic automorphic representations}\label{sec:locallyanalytic basic}
In the automorphic side of the $p$-adic Langlands, the locally analytic representations of $GL_2(\Q_p)$ were initially studied by Schneider and Teitelbaum in \cite{Sch10}, \cite{SchneiderBan} and \cite{Algebras}.  In the following, we are going to first recall some basic general facts about locally analytic representations. Then, we will state results of Berger, Breuil and Colmez on locally analytic vectors of $\mathbf{\Pi}(V)$ or $\mathbf{B}(V)$. 

Let $L$ be an extension of $K$ over $\Q_p$, $G$ be a locally $K$-analytic group, $U$ be a $L$-Banach space representation of $G$.  Let $\scriptC^{\la}(G,U)$ be the space of locally analytic $U$-valued functions on $G$. The space of locally analytic vectors of $U$ is the subspace of $U$ consisting of those vectors $u \in U$ for which the orbit map $g \mapsto gu$ is in $\scriptC^{\la}(G,U)$. We denote this space of locally analytic vectors by $U_{\la}$ (unlike the notation of Schneider and Teitelbaum, here we have followed the notation of Emerton  \cite{Emertonbook} and used `la' as a subscript to denote locally analytic vectors). We say that $U$ is a locally analytic $G$-representation if the natural map $U_{\la} \rightarrow U$ is a bijection. (This definition can be generalized to  barrelled locally convex Hausdorff spaces \cite{Emertonbook} but we will mostly be interested in locally analytic vectors of  Banach spaces $\mathbf{B}(V)$). Let $\scriptD^{\la}(G,L)$ be the locally analytic distribution algebra on $G$ (strong dual of $\scriptC^{\la}(G,L)$, i.e.  $\Hom_L(\scriptC^{\la}(G,L),L)_b$ where the subscript `b' denotes strong topology \cite{NFA}). The crucial property of $\scriptD^{\la}(G,L)$ is that it is a Fréchet-Stein algebra when $G$ is compact \cite[Thm. 5.1]{Algebras}. That is, $$\scriptD^{\la}(G,L)=\varprojlim_r\scriptD_r(G,L)$$ where $\scriptD_r(G,L)$ are $L$-Banach noetherian  algebras with flat transition maps. This allowed Schneider and Teitelbaum to describe a category of coadmissible modules over $\scriptD^{\la}(G,L)$. A coadmissible module $M$ will have the form $M \cong \varprojlim_rM_r$ where each $M_r$ is a finitely generated $D_r(G,L)$ module (carrying a Banach topology induced from $\scriptD_r(G,L)^n \rightarrow M$) together with compatible isomorphisms $$\scriptD_r(G,L) \otimes_{\scriptD_{r^{\prime}}(G,L)}M_{r^{\prime}} \cong M_r \quad \text{for } r^{\prime}<r.$$ We equip $M$ with the projective limit topology which makes it a $L$-Fréchet space. In the case when $G$ is compact, an admissible locally analytic $G$-representation over $L$ is a locally analytic $G$-representation on an $L$-vector space of compact type $U$ such that the strong dual $U^{\prime}_b=\Hom_L(U, L)_b$ is a coadmissible $\scriptD^{\la}(G,L)$-module equipped with its canonical topology. For general $G$, a locally analytic $G$-representation over $L$ is called admissible if it is admissible as an $H$ representation for one (equivalently any) open compact subgroup $H$ of $G$. If $G$ is compact, $U$ is a locally analytic $G$-representation, then $U$ is strongly admissible if $U^{\prime}_b$ is finitely generated as a $\scriptD^{\la}(G,L)$-module (as $U^{\prime}_b$ is quotient of some $\scriptD^{\la}(G,L)^n$ by some closed submodule, by Lemma $3.6$ of \cite{Algebras}, it is coadmissible with its canonical topology).  Emerton \cite[Thm. 5.1.15]{Emertonbook} further showed that being strongly admissible is equivalent to the fact that $U$ admits a closed $G$-equivariant embedding into $\scriptC^{\la}(G,L)^n$ for some natural number $n$ (which comes by dualizing the surjection $\scriptD^{\la}(G,L)^n \twoheadrightarrow U_b^{\prime}$). Here is an important theorem of Schneider and Teitelbaum \cite[Thm. 7.1]{Algebras} which finds locally analytic vectors of a $L$-Banach representation of $G$.
\begin{theorem}
	Suppose $G$ is compact and $U$ is a $L$-Banach representation of $G$ such that the dual $U^{\prime}_b$ is finitely generated module over the Iwasawa algebra $L[[G]]$  of $G$ (the algebra of $\Z_p$-valued measures on $G$, base changed to  $L$). Then
	\begin{enumerate}
		\item $U_{\la}$ is dense in $U$.
		\item $U_{\la}$ is strongly admissible $G$-representation.
		\item $(U_{\la})_b^{\prime} \cong \scriptD^{\la}(G,L) \otimes_{L[[G]]}U_b^{\prime}$.
	\end{enumerate}
\end{theorem}
\subsection{Locally analytic vectors of crystalline representations}\label{subsec: locallycrystalline}
With notations as in Section \ref{sec:padicLang}, we know that for a rank-two, irreducible, étale $\phigamma$-module $D$ over $\robba_L$, we have an exact sequence
 $$ 0 \rightarrow \mathbf{\Pi}(D)^{\ast} \otimes \delta_D \rightarrow D \boxtimes \mathbf{P}^1 \rightarrow \mathbf{\Pi}(D) \rightarrow 0.$$ 
Then Colmez found the locally analytic vectors of $\Pi(D)$ and showed that the following sequence is short exact.
$$0 \rightarrow (\mathbf{\Pi}(D)_{\la})^{\ast} \otimes \delta_D \rightarrow D_{\rig} \boxtimes \mathbf{P}^1 \rightarrow \mathbf{\Pi}(D)_{\la} \rightarrow 0,$$ 
and $(\check{\Pi}(D)_{\la})^{\ast} \cong D_{\rig}^\natural \boxtimes \mathbf{P}^1 \cong (\mathbf{\Pi}(D)_{\la})^{\ast} \otimes \delta_D$ (see. Theorem 0.17 and Theorem V.2.20 of \cite{Colmezphi}).
Now suppose $V$ is irreducible, crystalline, Frobenius semisimple with distinct Hodge-Tate weights $0$ and $k-1$. Following notation as in Section \ref{sec:padicLang}, Berger and Breuil defines a locally analytic principal series 
\begin{equation}\label{eq:LAalpha}
LA(\alpha)=\big(\Ind_{B(\Q_p)}^{GL_2(\Q_p)} \unr(\alpha^{-1}) \otimes x^{k-2}\unr(p\beta^{-1})   \big)_{\la}. 
\end{equation}
 We set $LA(\beta)$ by replacing $\alpha$ with $\beta$. There are $GL_2(\Q_p)$-equivariant injections $LA(\alpha) \hookrightarrow B(\alpha)$ and $LA(\beta) \hookrightarrow B(\beta)$. Obviously the locally algebraic representations $\pi(\alpha)$ and $\pi(\beta)$ have closed $GL_2(\Q_p)$-equivariant inclusions 
\begin{equation}\label{embeddingsClosed}
\pi(\alpha) \hookrightarrow LA(\alpha) \quad \text{and} \quad \pi(\beta) \hookrightarrow LA(\beta).
\end{equation}

 Then Berger and Breuil (see \cite{BerBre}, \cite[Sec. 7.2]{BerBreuilNotes}) construct a natural $GL_2(\Q_p)$-equivariant map  from the amalgamated direct sum  $$LA(\alpha) \oplus_{\pi(\beta)}LA(\beta) \rightarrow \mathbf{B}(V)_{\la}$$ and conjecture that it is an isomorphism if $\val_p(\alpha)<k-1$. This conjecture is now proved by Liu \cite[Thm. 4.1]{Liu}. More precisely Liu constructs a map 
 \begin{equation}\label{map:F}
 \mathcal{F}: (LA(\alpha) \oplus_{\pi(\beta)}LA(\beta))^{\ast} \rightarrow D_{\rig}^{\natural}(\check{V})\boxtimes \mathbf{P}^1
 \end{equation}

and shows that this is an isomorphism. This proves that
\begin{equation}
LA(\alpha) \oplus_{\pi(\beta)}LA(\beta) \cong \mathbf{B}(V)_{\la}.
\end{equation}
Further by Liu, we have a $GL_2(\Q_p)$-equivariant commutative diagram

\[	
\begin{CD}
(B(\alpha)/L(\alpha)=\mathbf{B}(V))^{\ast}     @>{\cong}>>  D^{\natural}(\check{V})\boxtimes \mathbf{P}^1\\
@VV{~}V       @VV{~}V\\
(LA(\alpha) \oplus_{\pi(\beta)}LA(\beta))^{\ast}   @>{\cong}>>  D_{\rig}^{\natural}(\check{V})\boxtimes \mathbf{P}^1
\end{CD}
\]
\begin{remark}
	Note that the locally analytic principal series $LA(\alpha)$ and $LA(\beta)$  are admissible locally analytic representations (cf. Example $1.18$ of \cite{Liu}).
\end{remark}
\section{Rigid analytic vectors of automorphic representations}\label{sec:rigidbasic}
In this section, we recall the basic notions in the theory of rigid (globally) analytic  $p$-adic representations.  In the following, we will follow the exposition of \cite{Emertonbook}. We assume basic notions of rigid analytic geometry from \cite{Bosch1} and non  archimedean functional analysis from \cite{NFA}.
\begin{definition}
	Let $U$ be a Hausdorff locally convex topological $L$-vector space. We say that $U$ is a $FH$-space if it admits a complete metric that induces a locally convex topology on $U$ finer than its given topology. We refer to the topological vector space structure on $U$ induced by such a metric as a latent Fréchet space structure on $U$. If furthermore, this latent Fréchet structure is defined by a norm, then we say that $U$ is a $BH$-space. 
\end{definition}
Any $L$-Banach space is, of course, a $BH$-space. Let $\XX$ be an affinoid  rigid analytic space defined over $K$. Let $\CA(\XX,L)$ denote the $L$-Banach algebra of $L$-valued rigid analytic functions defined on $\XX$. If $U$ is a $L$-Banach space, then we define the $L$-Banach space $\CA(\XX,U)$ of $U$-valued rigid analytic functions on $\XX$ to be the completed tensor product $\CA(\XX,L) \widehat{\otimes}_L U$. If $U$ is a Hausdorff locally convex topological $L$-vector space, then we define the locally convex space $\CAXU$ of $U$-valued rigid analytic functions on $\XX$ to be the locally convex inductive limit of Banach spaces $$\CAXU:=\varinjlim \CA(\XX, \closure{W}),$$
where $W$ runs over the directed set of all $BH$-subspaces $W$ of $U$. Furthermore, if $U$ is Fréchet, then also $\CAXU \cong \CA(\XX,L) \widehat{\otimes}_L U$ \cite[Prop. 2.1.13]{Emertonbook}. We say that $\XX$ is $\sigma$-affinoid if there is an increasing sequence $\XX_1 \subset \XX_2 \subset \cdots \subset \XX_n \subset \cdots $ of affinoid open subsets of $\XX$ such that $\XX = \cup_{n=1}^{\infty}\XX_n$ with $\{\XX_n\}_{n \geq 1}$ forming an admissible cover of $\XX$. Of course, an affinoid space is $\sigma$-affinoid. The basic example of a $\sigma$-affinoid rigid-analytic space, which is not affinoid, is an open ball which is a union of an increasing sequence of closed balls. (This notion of $\sigma$-affinoid space will be useful in our next section where we will find our rigid analytic vectors of $\mathbf{B}(V)$ under the action of a $\sigma$-affinoid rigid analytic subgroup of $GL_2(\Z_p)$). If $\XX$ is $\sigma$-affinoid and $U$ is a Hausdorff locally convex $L$-vector space then we define the convex $L$-vector space $\CAXU$ to be the projective limit $\varprojlim_{\mathds{Y}}\CA(\mathds{Y},U)$ where $\mathds{Y}$ runs over all admissible affinoid open subsets of $\XX$. If $U$ is a Fréchet
space and $\XX$ is $\sigma$-affinoid, then Emerton   also  shows that $\CAXU \cong \CA(\XX,L) \widehat{\otimes}_L U$ \cite[Prop. 2.1.19]{Emertonbook}.

Suppose $\GG$ is an affinoid rigid analytic group and suppose that the group of $K$-rational points $G:=\GG(K)$ is Zariski dense in $\GG$. Let $U$ be a $L$-Banach space with a topological action of $G$. We define the $\GG$-(rigid) analytic vectors $\UGan$ of $U$ to be the subspace of those vectors $u$ for which the orbit map $g \mapsto gu$ is an element of $\CAGU$. If $U$ is arbitrary convex $L$-vector space, then $\UGan=\varinjlim_W\closure{W}_{\GG-\an}$ where $W$ runs over all $G$-invariant $BH$-subspaces of $U$. If $U$ is Fréchet, then $\UGan \rightarrow \CAGU$ is a closed embedding and, in particular, $\UGan$ is again a Fréchet space. As before, if $\GG$ is $\sigma$-affinoid and $U$ is Hausdorff locally convex $L$-vector space equipped with a topological $\GG$-action, then we define $\UGan:=\varprojlim U_{\mathds{H}-\an}$, where the projective limit is taken over all admissible affinoid open subgroups $\mathds{H}$ of $\GG$. In particular, if $\mathds{H}_1 \subset \mathds{H}_2$ is an inclusion of admissible open affinoid subgroups of $\GG$, then the natural map $U_{\mathds{H}_2-\an} \rightarrow U_{\mathds{H}_1-\an} $ is an injection, because when composed with the natural injection $U_{\mathds{H}_1-\an}  \rightarrow U$, yields the natural injection $U_{\mathds{H}_2-\an}  \rightarrow U$. The representation $U$ is called a rigid analytic $\GG(K)$-representation if the natural inclusion map $\UGan \rightarrow U$ is a bijection.
\subsection{Passage from rigid analytic to locally analytic vectors}
Suppose $\GG$ is a locally $K$-analytic group. Suppose that $(\phi, H, \mathds{H})$ is a chart of $G$, that is $H$ is a compact open subgroup of $G$, $\mathds{H}$ is an affinoid rigid analytic space over $L$ isomorphic to a closed ball and $\phi: H \xrightarrow{\cong}\mathds{H}(K)$ with the additional property that $H$ is a subgroup of $G$. We refer to such a chart an analytic open subgroup. Let $U$ be a Hausdorff locally convex topological $L$-vector space equipped with a topological $G$-action, then for each analytic open subgroup $H$ of $G$ we can form the convex space $U_{\mathds{H}-\an}$. Then the locally analytic vectors of $U$ is the locally convex inductive limit $$U_{\la}= \varinjlim_H U_{\mathds{H}-\an},$$
where $H$ runs over all analytic open subgroups of $G$. Suppose $U=\varinjlim_n U_n$ is a locally analytic representation of $G$, where $U_n$ are $L$-Banach representations of $G$ with injective and compact transition maps. Then Emerton shows that $U$ is admissible (in the sense of Schneider and Teitelbaum) if and only if $(U_{\mathds{H}-\an})_b^{\prime}$ is finitely generated as a $\mathcal{D}^{\an}(\mathds{H},L)$-module for any cofinal sequence of analytic open subgroups $\mathds{H}$ \cite[Chapter 6]{Emertonbook}. Here $\mathcal{D}^{\an}(\mathds{H},L)$ is the rigid analytic distributions on $\mathds{H}$ (the dual of $\CA(\mathds{H},L)$).

\section{Wide open congruence subgroups and distribution algebra}\label{sec:wideopen}
Emerton introduced the notion of good analytic open subgroups in \cite[Sec. 5.2]{Emertonbook}. In the case of congruence subgroups of $GL_2(\Z_p)$, these are termed as wide open by Patel, Schmidt and Strauch \cite{PSS}. In the following we follow the presentation in \cite{PSS} and we recall few properties of rigid analytic distribution algebras of wide open congruence subgroups. Let $n \geq 0$ be a non-negative integer. We set $$\Gbf(0)=\Gbf = GL_2(\Z_p)= \Spec(\Z_p[a,b,c,d,1/\Delta]),$$ where $\Delta=ad-bc$; comultiplication given by the usual formula. Define the affine group scheme $\Gbf(n)$ over $\Z_p$ by setting $\OO(\Gbf(n))=\Z_p[a_n,b_n,c_n,d_n,1/\Delta_n].$ Here $$\Delta_n=(1+p^na_n)(1+p^nd_n)-p^{2n}b_nc_n,$$ where $a_n, b_n, c_n, d_n$ denote indeterminants. The group scheme homomorphism $\Gbf(n) \rightarrow \Gbf(n-1)$ given on the level of algebras are $$a_{n-1} \mapsto pa_n, b_{n-1} \mapsto pb_n, c_{n-1} \mapsto pc_n, d_{n-1} \mapsto pd_n,$$
if $n >1$. For $n=1$, we set $$a \mapsto 1+pa_1, b \mapsto pb_1,c \mapsto pc_1,d \mapsto 1+pd_1.$$
If $R$ is a flat $\Z_p$-algebra, the homomorphism $\Gbf(n) \rightarrow \Gbf(0)=\Gbf$ gives an isomorphism $\Gbf(n)(R)$ with $\Gbf(R)$ and $$\Gbf(n)(R)=\Big\{ \left( \begin{matrix} a & b \\  c & d \end{matrix} \right) \in \Gbf(R) \mid a-1,b,c,d-1 \in p^nR     \Big\},$$
where we make a formal identification by setting $a=1+p^na_n, b=p^nb_n, c=p^nc_n$ and $d=1+p^nd_n$.  Let $\widehat{\Gbf}(n)$ be the completion of $\Gbf(n)$ along its special fiber $\Gbf(n)_{\mathbb{F}_p}$ which is a formal group scheme over $\mathrm{Spf}(\Z_p)$. We denote its generic fiber in the sense of rigid analytic geometry \cite{Bosch1} by $\GG(n)$. This is an affinoid rigid analytic group. For a finite extension $K$ over $\Q_p$, its $K$-points  are 
\begin{equation}\label{eq:RigidCongKer}
\GG(n)(K)= \Big\{ \left( \begin{matrix} a & b \\  c & d \end{matrix} \right)  \mid a-1,b,c,d-1 \in p^n\OO_K    \Big\}.
\end{equation}
Let $\widehat{\Gbf(n)}^{\circ}$ be the completion of $\Gbf(n)$ in the closed point corresponding to the unit element in $\Gbf(n)_{\mathbb{F}_p}$. This is again a formal scheme over  $\mathrm{Spf}(\Z_p)$. Its generic fiber is the so-called `wide open' rigid analytic group which we denote by $\Gnnot$. We have $$\Gnnot(K)= \Big\{ \left( \begin{matrix} a & b \\  c & d \end{matrix} \right)  \mid a-1,b,c,d-1 \in p^n\mathfrak{m}_{\OO_K}   \Big\}$$
We can identify $\Gnnot$ with the rigid analytic four dimensional open polydisc $(\mathds{B}^{\circ})^4$ (which is a $\sigma$-affinoid space) via coordinates of the second kind
$$(t_1,t_2,t_3,t_4) \mapsto \mathrm{exp}(t_1p^ne)\mathrm{exp}(t_2p^nh_1)\mathrm{exp}(t_3p^nh_2)\mathrm{exp}(t_4p^nf),$$
where $$e=\left( \begin{matrix}0 & 1 \\  0 & 0 \end{matrix} \right), h_1= \left( \begin{matrix} 1 & 0 \\  0 & 0 \end{matrix} \right), h_2=\left( \begin{matrix} 0 & 0 \\  0 & 1 \end{matrix} \right), f= \left( \begin{matrix} 0 & 0 \\  1 & 0 \end{matrix} \right)$$
are elements of the enveloping algebra $U(\mathfrak{g})$ with $\mathfrak{g} = \Lie(GL_2(\Q_p))$.  Then,  we can identify functions on $\Gnnot$ as functions on $(\mathds{B}^{\circ})^4$ via pull-back. This justifies that $\Gnnot$ is `good analytic' in the sense of Emerton \cite[Sec. 5.2]{Emertonbook}. By Proposition $5.2.6$ of \cite{Emertonbook} (see also Remark 2.4.4 of \cite{PSS}), we know that the distribution algebra $\DanGnnotL$ is of compact type and can be written as an inductive limit 
$$\DanGnnotL = \varinjlim_m \DanGnnotL^{(m)}$$ where each $\DanGnnotL^{(m)}$ 
is a noetherian $L$-Banach algebra; hence $\DanGnnotL$ is a coherent ring. This coherent property of $\DanGnnotL$ makes it interesting to study the category of finitely presented modules (which now becomes an abelian category) over $\DanGnnotL$. Following ideas of Schneider and Teitelbaum of studying coadmissible modules over the locally analytic distribution algebra of compact $p$-adic Lie groups, Emerton, Patel, Schmidt and Strauch have also studied the category of modules $M$ over $\DanGnnotL$ which are of the form $M\cong \varinjlim_m M_m$ where each $M_m$ is finitely generated over the noetherian ring $\DanGnnotL^{(m)}$ with compatibility condition $$M_{m^{\prime}} \cong \DanGnnotL^{(m^{\prime})} \otimes_{\DanGnnotL^{(m)}} M_m \quad \text{ for } m^{\prime} >m,$$ (see Lemma $A.6$ of \cite{EmertonJac} and Section $5$ of \cite{PSS2}). Finally, we jot down some useful properties of $\DanGnnotL$ which will be useful in the next section when we find $\Gnnot$-analytic vectors of $\mathbf{B}(V)$ when $V$ is irreducible, crystalline, Frobenius semisimple with distinct Hodge-Tate weights. The following facts can be found in the Appendix of Emerton's paper \cite{EmertonJac}. 
\begin{proposition}\label{prop:finitepresented}
	If $U$ is an admissible locally analytic $GL_2(\Z_p)$-representation over $L$. Then 
	\begin{enumerate}
		\item $(U_{\Gnnot-\an})^{\prime}$ is finitely presented as a $\DanGnnotL$-module.
		\item There is a natural isomorphism $(U_{\Gnnot-\an})^{\prime}  \cong \DanGnnotL \widehat{\otimes}_{\mathcal{D}^{\la}(\Gnnot(\Q_p),L)} U^\prime$.
	\end{enumerate}
\end{proposition}
\begin{proposition}\label{prop:exactness}
	 The functor $U \longmapsto U_{\Gnnot-\an}$ on the category of admissible locally analytic $GL_2(\Z_p)$-representations is exact in the strong sense (that is, it takes an exact sequence of admissible locally analytic representations to a strict exact sequence of nuclear Fréchet spaces).
\end{proposition}
\section{Rigid analytic vectors of crystalline representations}\label{sec:mainheart}
Suppose that $V$ is irreducible, crystalline, Frobenius semisimple with distinct Hodge-Tate weights $0$ and $k-1$. With notations as in Section \ref{subsec: locallycrystalline}, $$\mathbf{B}(V)_{\la} \cong LA(\alpha) \oplus_{\pi(\beta)} LA(\beta).$$
Our goal in this section is to find out (and provide an explicit description of) the $\Gnnot$-analytic vectors of $\mathbf{B}(V)_{\la}$, denoted by $\big(\mathbf{B}(V)_{\la}\big)_{\Gnnot-\an}.$ Let us first search for $\Gnnot$-analytic vectors of $LA(\alpha)$.  Recall that the locally analytic induced principal series $LA(\alpha)$ is $ \big(\Ind_{B(\Q_p)}^{GL_2(\Q_p)} \chi_{\alpha}  \big)_{\la}$ where $\chi_{\alpha}= \unr(\alpha^{-1}) \otimes x^{k-2}\unr(p\beta^{-1}) $. The Iwasawa decomposition \cite[Sec. 3.2.2]{Orlik1} gives
\begin{equation}\label{iwasawaDecomposition}
 \big(\Ind_{B(\Q_p)}^{GL_2(\Q_p)} \chi_{\alpha}  \big)_{\la} \cong \big(\Ind_{B(\Z_p)}^{GL_2(\Z_p)} \chi_{\alpha}  \big)_{\la}
 \end{equation}
  
as $GL_2(\Z_p)$-equivariant topological isomorphism. Let $I(\Z_p)$ be the Iwahori subgroup $\left( \begin{matrix} \Z_p^\times & p\Z_p \\  \Z_p & \Z_p^\times \end{matrix} \right),$ $W$ be the ordinary Weyl group of $GL_2(\Q_p)$ with respect to $B(\Z_p)$, $P_w(\Z_p)=I(\Z_p) \cap wB(\Z_p)w^{-1}$ for $w \in W$. By the Bruhat-Tits decomposition (\textit{loc.cit} and \cite[Sec. 3.5]{Cartier1}), $$GL_2(\Z_p) = \sqcup_{w \in W} I(\Z_p)wB(\Z_p),$$
we obtain the decomposition 
\begin{equation}\label{eq:decomposition Weyl}
\big(\Ind_{B(\Z_p)}^{GL_2(\Z_p)} \chi_{\alpha}  \big)_{\la}\cong \oplus_{w \in W}\big(\Ind_{P_w(\Z_p)}^{I(\Z_p)}(\chi_{\alpha}^w)\big)_{\la},
\end{equation}
an $I$-equivariant decomposition of topological vector spaces, where the action of $\chi_{\alpha}^w$ is given by $\chi_{\alpha}^w(h)=\chi_{\alpha}(w^{-1}hw)$ with $h \in P_w(\Z_p)$. 

%The action of $GL_2(\Z_p)$ on $\big(\Ind_{B(\Z_p)}^{GL_2(\Z_p)} \chi_{\alpha}  \big)_{\la}$ is given by left translation. 

We first consider $w =Id$ and find out $\Gnnot$-analytic vectors in $\big(\Ind_{P(\Z_p)}^{I(\Z_p)}(\chi_{\alpha})\big)_{\la}$; here $P(\Z_p)$ is $P_{Id}(\Z_p)$.  In the following, since we are always over $\Z_p$, by abuse of notation, we will write $I$ for $I(\Z_p)$ and $P$ for $P(\Z_p)$.
The space $$\big(\Ind_P^I(\chi_{\alpha})\big)_{\la}:=\Big\{   f \in  \scriptC^{\la}(I,L): f(gb)=\chi_{\alpha}^{-1}(b)f(g) \Big\},$$
where $g \in I, b \in P$. The action of $I$ on $\big(\Ind_P^I(\chi_{\alpha})\big)_{\la}$ is given by left translation. Let $I(1)$ be the pro-$p$ Iwahori subgroup of $I$. Note that, since $\chi_{\alpha}$ is fixed, the restriction of the functions of $\big(\Ind_P^I(\chi_{\alpha})\big)_{\la}$ to $I(1) \subset I$ is injective. Therefore, we see that the space of $\big(\Ind_P^I(\chi_{\alpha})\big)_{\la}$ is isomorphic to 
$$\big(\Ind_Q^{I(1)}(\chi_{\alpha})\big)_{\la}:=\Big\{   f \in  \scriptC^{\la}(I(1),L): f(gb)=\chi_{\alpha}^{-1}(b)f(g) \Big\},$$
where $Q = P \cap I(1)$. The pro-$p$ Iwahori $I(1)$ has a natural triangular decomposition $I(1)=UQ$ where $ U =  \Big\{\left(\begin{matrix}
1 & 0 \\ z & 1
\end{matrix}\right), z \in \Z_p\Big\}.$
As a vector space, $\big(\Ind_Q^{I(1)}(\chi_{\alpha})\big)_{\la}$ is isomorphic to $\scriptC^{\la}(U,L)\cong \scriptC^\la(\Z_p,L)$. This later identification is given by $f\left(\begin{matrix}
1 & 0 \\ z & 1
\end{matrix}\right) \mapsto f(z)$ where $z \in \Z_p$, identified as points of the rigid analytic closed affinoid unit ball $\mathds{B}^1$. The action of $I(1)$ on $\scriptC^\la(\Z_p,L)$ is given by the following lemma.
\begin{lemma}\label{lem:action}\cite[Lemma 3.3]{Clozel_global_II_published}
	For $y \in \Z_p, x \in p\Z_p, s, t \in 1 +p\Z_p$, we have
	\begin{enumerate}[(i)]
		\item $\left(\begin{matrix}
			1 & 0 \\y & 1
		\end{matrix}\right) f(z) =f(z-y)$,
		\item $\left(\begin{matrix}
			s & 0 \\ 0 & t
		\end{matrix}\right) f(z) =f(st^{-1})\chi_{\alpha}(s,t),$
		\item $\left(\begin{matrix}
		1 & x \\ 0 & 1
		\end{matrix}\right)f(z)=f(\frac{z}{1-xz})\chi_{\alpha}((1-xz)^{-1},1-xz). $
	\end{enumerate}
\end{lemma}
We want to find the subspace  $\big(\Ind_Q^{I(1)}(\chi_{\alpha})_{\la}\big)_{\Gnnot-\an}$ where $\Gnnot$ is a $\sigma$-affinoid rigid analytic space defined in Section \ref{sec:wideopen}. Therefore, by Section \ref{sec:rigidbasic}, we have
$$\big(\Ind_Q^{I(1)}(\chi_{\alpha})_{\la}\big)_{\Gnnot-\an} \cong \varprojlim_{m>n}\big(\Ind_Q^{I(1)}(\chi_{\alpha})_{\la}\big)_{\mathds{G}(m)-\an} $$
Let $\UU(m)$ and $\UU$ be the rigid analytic affinoids such that $U(m):=\UU(m)(\Z_p) =\GG(m)(\Z_p) \cap \UU(\Z_p)$ where $\UU(\Z_p)=U$.
(Here $\GG(m)$ is  as in equation \eqref{eq:RigidCongKer}). Note that, we have realized $\big(\Ind_Q^{I(1)}(\chi_{\alpha})\big)_{\la}$ as a space of locally analytic functions on $U$. Let $ f \in \scriptC^\la(U,L)$. Suppose $f \in \big(\scriptC^\la(U,L)\big)_{\UU(m)-\an}$, then $u \mapsto uf \in \scriptC^\an(\UU(m),\scriptC^\la(U,L)) $ where $ u \in U(m)$. So $u \mapsto uf(1)$ is in $\scriptC^\an(\UU(m),L)$, because evaluation at $1 \in U(m)$ is a continuous map. Now, by part (i) of Lemma \ref{lem:action}, this gives that $u \mapsto f(u)$  is in $\scriptC^\an(\UU(m),L)$. Hence 
$\big(\Ind_Q^{I(1)}(\chi_{\alpha})_{\la}\big)_{\mathds{G}(m)-\an} $ can be identified with a subspace of $\scriptC^\la(U,L) \cap \scriptC^\an(\UU(m),L)$. In the following we will show that any member of 
$\scriptC^\la(U,L) \cap \scriptC^\an(\UU(m),L)$ is actually $\GG(m)$-analytic by using Lemma \ref{lem:action}.
\begin{proposition}\label{prop:UmAnalytic}
	We have $\big(\Ind_Q^{I(1)}(\chi_{\alpha})_{\la}\big)_{\mathds{G}(m)-\an} \cong \scriptC^\la(U,L) \cap \scriptC^\an(\UU(m),L)$
\end{proposition}

 Note that the space $ \scriptC^\an(\UU(m),L)$ is an $L$-Banach space with the following norm. Suppose $f \in  \scriptC^\an(\UU(m),L)$, then  
\begin{equation}\label{eq:Fanalytic}
f(z)=\sum_{l=0}^\infty a_lz^l, \quad a_l \in L,
\end{equation}

 is convergent in the ball $p^m\Z_p$. That is,
 \begin{equation}\label{eq:FanalyticMeaning}
  \val_p(a_l) + ml \rightarrow \infty  \text{ as } l \rightarrow \infty. 
 \end{equation}
 Then, we define the  $\scriptC^\an(\UU(m),L)$-valuation of $f$ by $\val_{\scriptC}(f)= \inf_l\{\val_p(a_l) +ml\}$.
 \begin{proof}[Proof of Proposition \ref{prop:UmAnalytic}]
 Consider the relation $\left(\begin{matrix}
 1 & 0 \\y & 1
 \end{matrix}\right) f(z) =f(z-y)$ where $y \in p^m\Z_p$ and $f \in \scriptC^\an(\UU(m),L)$.
 Then, 
 \begin{align*}
 f(z-y) &= \sum_{l=0}^\infty a_l(z-y)^l = \sum_{l=0}^\infty a_l \sum_{v=0}^l{l \choose v}z^{l-v}(-1)^vy^v\\
 &=\sum_{v=0}^\infty y^v\Big(\sum_{l \geq v}a_l{l \choose v} (-1)^vz^{l-v}\Big) =\sum_{v=0}^\infty y^vf_v.
 \end{align*}
For each fixed $v$, $\val_{\scriptC}(f_v) \geq \inf_{l \geq v}\{\val_p(a_l) +m(l-v)\}= \inf_{l \geq v}\{\val_p(a_l) +ml\}-mv$.  Therefore, by equation \ref{eq:FanalyticMeaning}, $$\val_{\scriptC}(f_v) +mv \geq  \inf_{l \geq v}\{\val_p(a_l) +ml\} \rightarrow \infty \text{ as } v \rightarrow \infty.$$ This implies that the action 
$\left(\begin{matrix}
1 & 0 \\y & 1
\end{matrix}\right) \mapsto \left(\begin{matrix}
1 & 0 \\y & 1
\end{matrix}\right) f$ is analytic on $\UU(m)$, i.e. it belongs to $ \scriptC^\an(\UU(m), \scriptC^\an(\UU(m),L))$.

Next consider the action in part (iii) of Lemma \ref{lem:action}. This is given by $$\left(\begin{matrix}
1 & x \\ 0 & 1
\end{matrix}\right)f(z)=f(\frac{z}{1-xz})\chi_{\alpha}((1-xz)^{-1},1-xz),$$ where now $x ,z \in p^m\Z_p$. Note that the character $\chi_{\alpha}= \unr(\alpha^{-1}) \otimes x^{k-2} \unr(p\beta^{-1})$, $\alpha, \beta \in \OO_L$; the unramified character $\unr(-): \Q_p^\times \rightarrow L^\times$ is given by $\unr(\lambda): y \mapsto \lambda^{\val_p(y)}$. Therefore, for $s, t \in 1+p\Z_p$, $$\chi_\alpha\left(\begin{matrix}
s & x \\ 0 & t
\end{matrix}\right)=t^{k-2}.$$
Hence $\chi_{\alpha}((1-xz)^{-1},1-xz)=(1-xz)^{k-2}$. This being a polynomial in $x$, is always rigid analytic in $x$ when $x$ varies over $\mathds{B}^1(p^m\Z_p)$. So we can forget about the factor $\chi_{\alpha}((1-xz)^{-1},1-xz)$ and it will be enough to show that the action $\left(\begin{matrix}
1 & x \\0 & 1
\end{matrix}\right) \mapsto f(\frac{z}{1-xz})$ is rigid analytic for the action of the upper unipotent rigid analytic subgroup of $\GG(m)$. Now,
\begin{align*}
f(\frac{z}{1-xz})&= \sum_{l=0}^\infty a_l\frac{z^l}{(1-xz)^l}= \sum_{l=0}^{\infty}a_lz^l\sum_{q=0}^\infty{l+q-1 \choose q}x^q z^q\\
&= \sum_{q=0}^\infty x^q \Big( \sum_{l=0}^{\infty}a_l{l+q-1 \choose q }z^{l+q}     \Big)= \sum_{q=0}^\infty x^q f_q,
\end{align*}
where we have used the fact that $(1-v)^{-l} = \sum_{q=0}^\infty {l+q-1 \choose q } v^q$ for $\lvert v \rvert_p <1$. For each $q$, 
\begin{align*}
\val_{\scriptC}(f_q) &\geq \inf_l\{\val_p(a_l) +m (l+q)\}\\
&= \inf_l\{\val_p(a_l) +ml\}+mq \\
&= \val_{\scriptC}(f) +mq
\end{align*}
Therefore, of course, $\val_\scriptC(f_q) +mq \geq \val_\scriptC(f) +2mq \rightarrow \infty \text{ as } q \rightarrow \infty$.

Next, for $ s \in 1+p^m\Z_p$ and $z \in p^m\Z_p$, consider 
\begin{equation*}
\left(\begin{matrix}
s & 0 \\0 & 1
\end{matrix}\right)f(z)=f(sz)\chi_\alpha(s,1)=f(sz).
\end{equation*}
Write $s = 1+s^{\prime}$, where $s^{\prime} \in p^m\Z_p$, and the analyticity has to be checked in the variable $s^\prime$. 
\begin{align*}
f(sz) &= \sum_{l=0}^\infty a_l(z+s^{\prime}z)^l= \sum_{l=0}^{\infty}a_l \sum_{q=0}^l{l \choose q} z^{l-q}(s^{\prime}z)^q\\
&=\sum_{q=0}^\infty (s^{\prime})^q\Big( \sum_{l \geq q}a_l{l \choose q} z^l     \Big)=\sum_{q=0}^\infty (s^{\prime})^qf_q.
\end{align*}
Therefore, $\val_{\scriptC}(f_q) \geq \inf_{l \geq q}\{\val_p(a_l) +ml\} \rightarrow \infty \text{ as } q \rightarrow \infty$ by equation \eqref{eq:FanalyticMeaning}. Hence, obviously, $\val_{\scriptC}(f_q) +mq \rightarrow \infty$ as $q \rightarrow \infty$. 
Finally, consider the action 
$$\left(\begin{matrix}
1 & 0 \\0 & t
\end{matrix}\right)f(z)=f\big(\frac{z}{t}\big)\chi_{\alpha}(1,t),$$
for $t=1+t^{\prime}, t^{\prime} \in p^m\Z_p$.  By the same argument as before, $\chi_{\alpha}(1,1+t^{\prime})$ is a polynomial in $t^{\prime}$ and so it is, of course, analytic when $t^\prime \in \mathds{B}^1(p^m\Z_p).$
\begin{align*}
f\big(\frac{z}{1+t^{\prime}}\big)&=\sum_{l=0}^\infty a_l\frac{z^l}{(1+t^{\prime})^l}= \sum_{l=0}^\infty a_lz^l\sum_{q=0}^\infty {l+q-1 \choose q} (-1)^q (t^\prime)^q\\
&=\sum_{q=0}^\infty (t^\prime)^q\Big( \sum_{l=0}^\infty a_l {l+q-1 \choose q } (-1)^q z^l     \Big)=\sum_{q=0}^\infty (t^\prime)^qf_q.
\end{align*}
Therefore, $\val_{\scriptC}(f_q) \geq \inf_l\{\val_p(a_l) +ml\}=\val_{\scriptC}(f)$. This implies that $ \val_\scriptC(f_q) +mq \rightarrow \infty$ as $q \rightarrow \infty$. This completes the proof 
of Proposition \ref{prop:UmAnalytic}.
\end{proof}
This gives the following corollary.
\begin{corollary}
	\begin{align*}
	\big(\Ind_Q^{I(1)}(\chi_{\alpha})_{\la}\big)_{\Gnnot-\an} &\cong \varprojlim_{m>n} \scriptC^\la(U,L) \cap \scriptC^\an(\UU(m),L) \\
	&\cong \scriptC^{\an}(\UU(n)^\circ, L) \cap \scriptC^\la(U,L).
	\end{align*}

\end{corollary}
Recall that the above Corollary gives us the analytic vectors when the Weyl element $w$ is identity. Having found out the analytic vectors for the Weyl orbit $w=Id$, a general argument in appendix A of \cite{Rayglobal} gives us the analytic vectors for all other non-trivial Weyl orbits coming from equation \eqref{eq:decomposition Weyl}.  Hence, it follows that 
\begin{align*}
\Big(\Ind_{P_w(\Z_p)}^{I(\Z_p)}(\chi_{\alpha}^w)_{\la}\Big)_{\Gnnot-\an} &\cong \varprojlim_{m>n} \scriptC^\an(w\UU(m)w^{-1},L) \text{ } \cap \text{ }  \scriptC^\la(wUw^{-1},L) \\
&\cong \scriptC^\an(w\UU(n)^\circ w^{-1},L)\text{ }  \cap \text{ }  \scriptC^\la(wUw^{-1},L),
\end{align*}

where $w\UU(m)w^{-1}$ and $w\UU(n)^{\circ}w^{-1}$ have usual meaning. (They are rigid analytic spaces such that $w\UU(m)w^{-1}(\Z_p)= w\UU(m)(\Z_p)w^{-1}$). 
Therefore, we have shown that the following Theorem holds.
\begin{theorem}\label{thm: Amalytic LA alpha beta}
Let $GA(\alpha)$ be the space of functions  $\oplus_{w \in W} \scriptC^\an(w\UU(n)^\circ w^{-1},L) \text{ } \cap \text{ } \scriptC^\la(wUw^{-1},L)$ where the action of the torus of $\UU(n)^\circ(\Z_p)$ on each of the direct sum is given by $\chi_{\alpha}^w$. Then $GA(\alpha)$ is stable  under the action of $\Gnnot(\Z_p)$ and is a rigid analytic $\Gnnot(\Z_p)$-representation in the sense of Emerton. Furthermore, we have 
	$$\big(LA(\alpha)\big)_{\Gnnot-\an} \cong GA(\alpha)$$	
	as rigid analytic $\Gnnot(\Z_p)$-representations.
\end{theorem}
Exactly, by a similar argument, replacing $\alpha$ by $\beta$ we can find the space $GA(\beta)$
of $\Gnnot$-analytic vectors of $LA(\beta)$.

Recall the locally algebraic representation $\pi(\beta)$ with a closed $GL_2(\Q_p)$-equivariant embeddding into $LA(\beta)$ (cf. equation \eqref{embeddingsClosed}), where 
$$\pi(\beta)= \Sym^{k-2}L^2 \otimes_L \big(  \Ind_{B(\Q_p)}^{GL_2(\Q_p)} \unr(\beta^{-1}) \otimes \unr(p\alpha^{-1})   \big)_\sm$$
(cf. equation \eqref{pibeta}). 
As before, by Iwasawa decomposition and Bruhat-Tits decomposition (cf. \eqref{iwasawaDecomposition} and \eqref{eq:decomposition Weyl}), we have 
\begin{equation*}
\big(\Ind_{B(\Q_p)}^{GL_2(\Q_p)} \unr(\beta^{-1}) \otimes \unr(p\alpha^{-1})   \big)_{\sm}\cong \oplus_{w \in W}\big(\Ind_{P_w(\Z_p)}^{I(\Z_p)}(\unr(\beta^{-1}) \otimes \unr(p\alpha^{-1}))^w\big)_{\sm},
\end{equation*}
Similar to the case of locally analytic vectors considered before, the smooth vectors of the space $\Ind_{P_w(\Z_p)}^{I(\Z_p)}(\unr(\beta^{-1}) \otimes \unr(p\alpha^{-1}))^w$ can be identified with locally constant functions on  $wUw^{-1}$. Therefore, the $\UU(m)$-analytic subspace  of $\big(\Ind_{P_w(\Z_p)}^{I(\Z_p)}(\unr(\beta^{-1}) \otimes \unr(p\alpha^{-1}))^w\big)_{\sm}$ is the set of functions which are constant on $w\UU(m)(\Z_p)w^{-1}$ and locally constant on its complement space, i.e. $wUw^{-1} \backslash w\UU(m)(\Z_p)w^{-1}$; we denote this space by $C(\beta)_{w\UU(m)w^{-1}}$. (The `$C$' stands for constant on $w\UU(m)w^{-1}$ and `$\beta$' means we are dealing with the character corresponding to $\beta$). 
Then the following corollary is obvious, as the polynomials $\Sym^{k-2}L^2$ are always, of course, rigid analytic.

\begin{comment}
Since the smooth vectors of induced principal series consist of locally constant functions with compact support, obviously the $\GG(m)$-analytic vectors of $ \big(  \Ind_{B(\Q_p)}^{GL_2(\Q_p)} \unr(\beta^{-1}) \otimes \unr(p\alpha^{-1})   \big)_\sm$ can be identified with  functions  which are constant on  $\UU(m)(\Z_p)$ and locally constant on its complement space, i.e. $U \backslash \UU(m)(\Z_p)$. (This is because `smooth' means `locally constant', but then we already know that $\UU(m)$-analytic subspace must lie in $\scriptC^\an(\UU(m),L)$).  Of course, the polynomials $\Sym^{k-2}L^2$ are rigid analytic. Therefore, we obtain the following corollary.
\end{comment}

\begin{corollary}\label{cor:pibetaAnalytic}
Let $C(\beta)_{m}$ denote $\oplus_{w \in W}C(\beta)_{w\UU(m)w^{-1}}$. 
Then $\Gnnot$-analytic vectors of $\pi(\beta)$ are given by 
$$\big(\pi(\beta)\big)_{\Gnnot-\an} \cong \Sym^{k-2}L^2 \otimes \varprojlim_{m>n}C(\beta)_{m}.$$
\end{corollary}
Replacing $\alpha$ by $\beta$ we obtain $\big(\pi(\alpha)\big)_{\Gnnot-\an} \cong \Sym^{k-2}L^2 \otimes \varprojlim_{m>n}C(\alpha)_{m}.$ 
From the discussion after equation $4.7$ of \cite[page 54]{Liu}, we note that 
\begin{equation}\label{eq:dualmap}
\big(LA(\alpha) \oplus_{\pi(\beta)} LA(\beta)\big)^\ast \cong \ker\Big(LA(\alpha)^\ast \oplus LA(\beta)^\ast \rightarrow \pi(\beta)^\ast\Big).
\end{equation}

But the locally analytic principal series $LA(\alpha)$ or $LA(\beta)$ are spaces of compact type, hence reflexive \cite[Prop. 16.10]{NFA}. Also closed subspace (and quotient by a closed subspace) of a reflexive space is  reflexive. Therefore, $\pi(\beta)$ is reflexive and  note that  the map $LA(\alpha)^\ast \oplus LA(\beta)^\ast \rightarrow \pi(\beta)^\ast$ in \eqref{eq:dualmap} is a surjection, since it comes from dualizing the natural injection $\pi(\beta) \hookrightarrow LA(\beta) \hookrightarrow LA(\alpha) \oplus LA(\beta)$.  Therefore, dualizing equation \eqref{eq:dualmap} and using Proposition \ref{prop:exactness}, we get the  following Theorem.
\begin{theorem}\label{thm:mainAnalytic}
Suppose $V=V(\alpha,\beta)$ is crystalline, irreducible, Frobenius semisimple Galois representation with distinct Hodge-Tate weights. Let $\mathbf{B}(V)_{\la}$ be the locally analytic vectors of the $GL_2(\Q_p)$-representation  $\mathbf{B}(V)$ associated to $V$ via $p$-adic local Langlands constructed by Berger and Breuil. Then,
$$\big(\mathbf{B}(V)_{\la}\big)_{\Gnnot-\an} \cong \Coker\Big( \pi(\beta)_{\Gnnot-\an} \hookrightarrow LA(\alpha)_{\Gnnot-\an} \oplus LA(\beta)_{\Gnnot-\an}  \Big)  $$
is a rigid-analytic $\Gnnot(\Z_p)$-representation. Furthermore,
\begin{enumerate}[(i)]
	\item $\pi(\beta)_{\Gnnot-\an} \cong \Sym^{k-2}L^2 \otimes \varprojlim_{m>n}C(\beta)_{m}$,
	\item   $ LA(\alpha)_{\Gnnot-\an} \cong GA(\alpha)$  and $ LA(\beta)_{\Gnnot-\an} \cong GA(\beta)$.
\end{enumerate}
In particular, it follows that $\big(\mathbf{B}(V)_{\la}\big)_{\Gnnot-\an} \neq 0$.
\end{theorem}
The fact that $\big(\mathbf{B}(V)_{\la}\big)_{\Gnnot-\an} \neq 0$ is obvious because we note that $\pi(\beta)$ and $\pi(\alpha)$ are isomorphic via the intertwining operator $I_{\sm}$ (cf. Lemma \ref{lemma:intertwining}) and we note that $\pi(\alpha)$ (resp. $\pi(\beta)$) are embedded into $LA(\alpha)$ (resp. $LA(\beta)$) (see \eqref{embeddingsClosed}).
But in Theorem \ref{thm:mainAnalytic}, we have only cut out the direct sum of $\Gnnot$-analytic vectors of $LA(\alpha) \oplus LA(\beta)$ by the $\Gnnot$-analytic vectors of $\pi(\beta)$. In particular, an isomorphic copy of $\pi(\alpha)_{\Gnnot-\an}$ certainly lives inside $\big(\mathbf{B}(V)_{\la}\big)_{\Gnnot-\an}$.

\section{Langlands base change of rigid analytic vectors  and future questions}\label{sec:basechange}
Let $K_0$ be a finite extension of $\Q_p$ and $K$ be an extension of $K_0$ of degree $r$.  Given a formal $\OO_{ K}$-scheme $\XX_{\OO_{ K}}$, Bertapelle constructs a Weil restriction functor which associates to $\XX_{\OO_{ K}}$ another formal scheme over $\OO_{K_0}$ \cite{Bertapelle1}.
Suppose $\XX_{K}$  is the rigid analytic space associated to the formal scheme $\XX_{\OO_{ K}}$. This gives a Weil restriction functor which associates to $\XX_{K}$ another rigid analytic space $\mathfrak{R}_{K/K_0}\XX_{K}$ over $K_0$ \cite[Prop. 1.8]{Bertapelle1}. 
We fix a choice of a basis of $K$ over $K_0$. In the following Lemma, we recall how Weil restriction behaves with affinoid closed unit balls. 
\begin{lemma}\cite[Lemma 1.1]{Clozel_global_II_published}
	Suppose $K$ is unramified over $K_0$. Let $\mathds{B}^1_K$ be the affinoid closed unit ball over $K$. Then,  $\frakRKKnot \mathds{B}^1_K$ is isomorphic to the $r$-th power of  $\mathds{B}_{K_0}$. Therefore,   $\frakRKKnot (\mathds{B}^1_K)^d \cong (\mathds{B}_{K_0})^{dr}$ as rigid analytic spaces.
\end{lemma}

Consider a rigid analytic group $\GG_K$ over $K$ isomorphic as a rigid analytic space to $ (\mathds{B}^1_K)^d $. 	Assume that $\GG_K$ is actually defined over $K_0$, i.e., is obtained by extension of scalars from $K_0$. Then $\scriptC^{\an}(\GG_K,K) \cong \scriptC^{\an}(\GG_{K_0}, K_0) \otimes K$. There is a natural comultiplication map $$m: \scriptC^{\an}(\GG_{K},K) \rightarrow \scriptC^{\an}(\GG_{K},K) \text{ } \widehat{\otimes} \text{ }  \scriptC^{\an}(\GG_{K},K)$$ which is defined by group multiplication in $\GG_K$. This comultiplication map  extends from the natural comultiplication map defined for $\GG_{K_0}$ and  defines a natural comultiplication map $$\mathfrak{R}(m): \scriptC^{\an}(\resRKKnot \GG_K, K)  \rightarrow \scriptC^{\an}(\resRKKnot \GG_K, K) \text{ }  \widehat{\otimes} \text{ }  \scriptC^{\an}(\resRKKnot \GG_K, K).$$
Here $\scriptC^{\an}(\resRKKnot \GG_K, K) = \scriptC^{\an}(\resRKKnot \GG_K, K_0) \otimes K$. Suppose $K$ is Galois over $K_0$. Therefore, the Galois group $\Gal(K/K_0)$ acts naturally on $\GG_K$ by $g$-linear automorphisms of the Tate algebra and acts on $\resRKKnot \GG_K$ by $K_0$-automorphisms. Then, Clozel showed the following Lemma relating rigid analytic functions on $\GG_K$ with those of $\resRKKnot \GG_K$ \cite[Sec. 1]{Clozel_global_II_published}.
\begin{lemma}\label{lemma: Clozel}
	\begin{enumerate}[(i)]
		\item There exists a natural map $b_1: \scriptC^{\an}(\GG_{K_0}, K) \rightarrow \scriptC^{\an}(\resRKKnot \GG_K, K)$ which commutes with natural comultiplication maps.
		\item There is a  natural map $$b= \prod_{g \in \Gal(K/K_0)}b_1^g: \scriptC^{\an}(\GG_{K_0}, K) \rightarrow \scriptC^{\an}(\resRKKnot \GG_K, K)$$ which commutes with comultiplication.
		\item Part (ii) induces an isomorphism  $$\scriptC^{\an}(\resRKKnot \GG_K, K) \cong \widehat{\otimes}_{g \in \Gal(K/K_0)}  \scriptC^{\an}(\GG_{K}, K).$$  The map $b$ is a tensor product 
		$b= \widehat{\otimes}_{g \in \Gal(K/K_0)} b_1^g. $  Its $g$-component sends a power series $\sum a_l\underline{x}^m, \underline{x}=(x_1,...,x_d)$ to the series $\sum a_lg(\underline{x})^m.$
		\item The maps $b$ and $b_1$ are continuous in canonical topologies (it respects sup norm on the Tate algebra of power series). 
	\end{enumerate}
\end{lemma}

Now suppose that $U$ is a rigid analytic  representation of $\GG_{K}(K_0)$ over $L$ where $\iota: K \hookrightarrow  L$ is the natural inclusion. Then, by \cite[Prop. 3.1]{Clozel_global_II_published}, $U$ extends naturally to a rigid analytic representation 
of $\GG_K(K)$. If $g \in \Gal(K/K_0)$, we have the injection $\iota \circ g: K \rightarrow L$. We write $U^g$ for the representation of $\GG_K(K)$ associated to $\iota \circ g$. It is $K$-analytic for $\iota \circ g$. Then, 
Clozel defines the  base change of $U$ compatible with Langlands functoriality.
\begin{definition}\label{def:basechange}\cite[Defn. 3.2]{Clozel_global_II_published}
	The Langlands base change of $U$ is the globally analytic representation of $\resRKKnot \GG_K(K)$ on $W= \widehat{\otimes}_{g \in \Gal(K/K_0)} U^{g}$.
\end{definition}
There are several future project that branches out from Theorem \ref{thm:mainAnalytic}, the main result of this article. We conclude this section by mentioning some of them. 
The first step is to  generalize Weil restriction functor for the case of rigid analytic $\sigma$-affinoid representations. Let us denote this functor by $\resRKKnot \GG_K^\circ$. This functor has to be constructed in such a way that it satisfies the natural  analogue of Clozel's result (cf. Lemma \ref{lemma: Clozel}). In particular, we must show that 
$$\scriptC^{\an}(\resRKKnot \GG_K^\circ, K) \cong \widehat{\otimes}_{g \in \Gal(K/K_0)}  \scriptC^{\an}(\GG_{K}^\circ, K).$$ 
Then we can construct the Langlands base change of $\big(\mathbf{B}(V)_{\la}\big)_{\Gnnot-\an} $   analogous to that of Clozel's, by showing that the corresponding  tensor product in Definition \ref{def:basechange} will be a globally analytic representation over the Weil restriction of sigma affinoid rigid analytic spaces. This is a future  project  and we will solve this in  a separate paper. 
\subsection{Why $\sigma$-affinoids and not affinoids?}\label{sec: whySigmaAffinoid}
In this article, the reader will notice that we compute rigid analytic vectors for $\sigma$-affinoid rigid analytic groups (open balls) instead of affinoids (closed balls). It  is a natural question to ask, why such a choice?
This is mainly due to Proposition \ref{prop:exactness}, that is, 	the functor $U \longmapsto U_{\Gnnot-\an}$ on the category of admissible locally analytic $GL_2(\Z_p)$-representations is exact. It is still an open question to study if $U \longmapsto U_{\GG(n)-\an}$ is exact. This seems to require some new idea. 
Further, the distribution algebra $\DanGnnotL$ is a coherent ring (Section \ref{sec:wideopen}) which makes it easy to study modules over this algebra. However, the distribution algebra of the affinoid group $\mathcal{D}^{\an}(\GG(n),L)$ is not so well behaved. More precisely, Clozel shows that the analytic distribution algebra of affinoids does not respect cartesian product,  $$\mathcal{D}^{\an}(X \times Y, L) \ncong \mathcal{D}^{\an}(X,L) \widehat{\otimes} \mathcal{D}^{\an}(Y,L),$$ where $X$ and $Y$ are rigid analytic spaces (cf. \cite[Appendix A.1]{Clozel_global_II_published}). Furthermore, $\mathcal{D}^{\an}(\GG(n),L)$ is not noetherian \cite[Appendix A.2]{Clozel_global_II_published}. Therefore, it remains open to determine the algebra structure of $\mathcal{D}^{\an}(\GG(n),L)$. 
\subsection{Representations non-crystalline.} In this article, we show the existence of rigid analytic vectors in $\mathbf{B}(V)_{\la}$ where $V$ is crystalline. A natural question is to ask if there exists rigid analytic vectors for non-crystalline  representations, that is, $V=V(s)$ when $s \in \scriptS_\ast^{\mathrm{ng}}  \sqcup \scriptS_\ast^{\mathrm{st}} \sqcup \scriptS_\ast^{\mathrm{ord}} \sqcup \scriptS_\ast^{\mathrm{ncl}}.$ This is too an open question which needs further research.

\section*{Achnowledgements}
The author heartily thanks Laurent Clozel for introducing him to the theory of rigid analytic $p$-adic representations and for suggesting to look at the problem discussed in this article.  The author also thanks him for suggesting the open question presented in this article. The author also benefited from numerous conversations with Filippo Nuccio on his visit to the University of British Columbia during September 2019 - August 2020.
The author is also grateful to Matthew Emerton for bringing his attention to Emerton's paper \cite{EmertonJac}.

\bibliographystyle{alpha}
%\bibliography{IwasawaPG}
\bibliography{Analytic}
\end{document}